\documentclass[11pt]{amsart}
 
\usepackage[margin=1in]{geometry}
\usepackage{amsmath, amsthm, amssymb, amsfonts}
\usepackage[hidelinks]{hyperref}
\usepackage{graphicx}
\graphicspath{{Images/}}
\usepackage{float}
\usepackage{tikz}
\usepackage{caption}
\usepackage{subcaption}
\usepackage{tikz-cd}
\usepackage{mathtools}
\usepackage{stmaryrd,wasysym,bm}
\usepackage{accents}
\usepackage{xcolor}

\title{On spaces of Euclidean triangles and triangulated Euclidean surfaces} 

\author{\.{I}smail Sa\u{g}lam, Ken'ichi Ohshika  and Athanase Papadopoulos}
\address{}
\email{}

\newtheorem{theorem}{Theorem}[section]
\newtheorem{corollary}[theorem]{Corollary}
\newtheorem{lemma}[theorem]{Lemma}
\newtheorem{proposition}[theorem]{Proposition}

\theoremstyle{definition}
\newtheorem{definition}[theorem]{Definition}
\theoremstyle{remark}
\newtheorem{remark}[theorem]{Remark}
\theoremstyle{note}

\theoremstyle{example}
\newtheorem{example}[theorem]{Example}

\newcommand{\R}{\mathbb R}

\numberwithin{equation}{section}
\numberwithin{table}{section}

\date{\today}
\begin{document}

\begin{abstract}

In this paper, we introduce an asymmetric distance function on the space of marked Euclidean triangles of normalised area, and we prove several properties of this metric, which turns out to be (a restriction of) a non-symmetric version of the classical Thompson distance.  We give a description of the geodesics of this metric, we show that it is Finsler, and we give a formula for its infinitesimal Finsler structure. We then introduce and study a Finsler metric of the space of singular Euclidean structures on a surface adapted to  an underlying fixed triangulation, and we also study its geodesics and its Finsler infinitesimal structure.
We then develop a theory of completeness and completion of asymmetric metrics which is adapted to our setting, and we use this theory in the study of the completeness of the metric we introduced on the space of triangles.  In doing so, we establish a bridgebetween one aspect of Thurston's theory of metrics on spaces of surfaces and Thompson's metrics.

\medskip 

	 	
		\medskip
		

\end{abstract}

\maketitle

\tableofcontents

\section{Introduction}
In this paper, we provide the space of Euclidean triangles with normalised area with an asymmetric Finsler metric, which turns out to be a normalised asymmetric Thompson distance. We obtain several geometrical results concerning this metric, including a description of its geodesics and its infinitesimal structure.  After the study of this metric on spaces of triangles, we introduce and study a Finsler metric of a space of singular Euclidean structures on a surface adapted to an underlying fixed triangulation. Like in the case of Euclidean triangles, we study the geodesics and the Finsler infinitesimal structure of this new metric.
We then develop a theory of completeness and completion for asymmetric metrics and we use it in the study of the completeness of the metric we introduce on the space of triangles. We address several questions on symmetrisations and on one-parameter deformations of such metrics.
 
 Our interest in studying distance functions on spaces of metrics on surfaces arose from reading Thurston's paper \emph{Minimal stretch maps between hyperbolic surfaces} \cite{thurston} in which he defined an asymmetric metric on Teichm\"uller spaces of hyperbolic surfaces based on a theory of best maps between such surfaces. 
This metric is now called Thurston's metric.  To the interested reader, we refer to the two recent surveys on this metric \cite{Xu} and \cite{Xu}
 Analogues of Thurston's theory in the setting of Euclidean surfaces have already been studied in the papers \cite{2005c, MOP2, Saglam1, Saglam2}. The present paper, which was first intended to be a sequel to these papers, explores the analogy with the Thompson distance. 

The plan of the rest of this paper is the following.

We start, in \S \ref{s:Thompson}, by introducing the asymmetric Thompson distance function defined on pairs of points in a cone in
a finite-dimensional real vector space. This is a generalised distance function: it does not satisfy the symmetry axiom,  it does not separate points) and it may take negative values. It is a building block  for several metrics which we shall encounter in the 
rest of the paper. Symmetrising the asymmetric Thompson distance by taking the supremum of the forward and backward distances between pairs of points gives the known (symmetric) Thompson distance.

In \S \ref{s:space}, we introduce an asymmetric metric $\eta$ on the space $\frak{T}_1$ of unit area triangles in the Euclidean plane parametrised by the three different expressions of the form $(a_i+a_j-a_k)/{2}$, where $a_1, a_2,  a_3$ are the edge lengths of such a triangle. 
We then give a necessary and sufficient condition for a path in our space of triangles to be a geodesic for this metric. This will allow us to prove that any two points in $\frak{T}_1$  can be joined by a geodesic. In fact, we show that any two points can be joined by a bigeodesic, that is, a geodesic which remains geodesic when it is traversed in the reverse direction. We give a characterisation for a path in this space to be geodesic. We then show that the metric $\eta$ is Finsler, and give a formula for the norm induced on each tangent space by this Finsler structure.

In \S \ref{s:surface},  we consider the case of compact surfaces endowed with singular Euclidean metrics arising from a fixed triangulation. We equip the set of marked such surfaces with an asymmetric metric, defined in a way that generalises the case of the space of metrics on a single triangle which we studied in the two preceding sections. 
We show that for any two points in the new space, there is a bigeodesic joining them.   We then show that this metric is Finsler, and give an explicit formula for the associated Finsler structure defined on the tangent space at each point. We also consider a family of Finsler metrics which contains as special cases the (asymmetric) metric itself and its arithmetic symmetrisation. 

In \S \ref{s:weighted}, we address some naturally arising questions  concerning two families of weighted metrics which generalise the metrics we studied in the previous sections.

  In \S \ref{s:completion}, we study the questions of completeness and completion for an asymmetric metric space, notions with which we deal later in the paper, namely, in the study of Teichm\"uller spaces of triangles and of triangulated surfaces equipped with singular flat metrics. We introduce notions of forward and backward Cauchy sequences, and a property which we call convergence-symmetry property for asymmetric metric spaces. In a convergence-symmetric metric space, a sequence is forward Cauchy  if and only if it is backward Cauchy. This allows us to talk simply of Cauchy sequences and completeness in metric spaces having the convergence-symmetry property. We show that in such a space, completeness is equivalent to the condition that the arithmetic symmetrisation, or, equivalently, the max symmetrisation, of the asymmetric metric, is complete. 
 
 In \S \ref{s:topology}, we study the topology induced by the asymmetric metric which we defined on the space of singular Euclidean metrics on a compact surface with a fixed underlying triangulation introduced in \S \ref{s:surface}. This topology is induced by the arithmetic or the max symmetrisation, and it leads to a natural notion of convergence of sequences in terms of the asymmetric metric.   
 
 In \S \ref{s:completeness}, we  show that the space of Euclidean   triangles satisfies the convergence-symmetry property for sequences, and that this space is complete in the sense defined in \S\S \ref{s:completion} and  \ref{s:topology}. We then give an example of a family of metrics on a 
 compact surface equipped with singular Euclidean metrics
with a fixed underlying triangulation which is neither convergence-symmetric nor forward complete.

\bigskip

\noindent {\it Acknowledgements.} We are grateful to the referee of the first version of this paper for his or her very careful reading and insightful and helpful comments, in particular, for pointing out the relation between the metric we studied on the space of Euclidean triangles and an asymmetric Thompson distance.

\section{The asymmetric Thompson distance} \label{s:Thompson}


In this section, we introduce the asymmetric Thompson distance  defined on a cone in a finite-dimensional real vector space. 
A symmetrisation of this distance is the original Thompson distance introduced by Thompson in \cite{Thom}.
Other references on the Thompson  distance include the book \cite{LemmensNussbaum2012} by Lemmens and Nussbaum, whose main purpose is
to generalise  the classical Perron--Frobenius theory of nonnegative matrices to nonlinear mappings in finite-dimensional spaces, and which contains at the same time an introduction to the  Thompson distance. See also the papers  \cite{Lemmens-Roelands, Nussbaum1988, Nussbaum-Walsh2004}, on the  Thompson distance. We shall recall this distance.

Let $V$ be a finite-dimensional real vector space endowed with its standard topology.
A nonempty subset $C$ of $V$ is called a \emph{cone} if it is convex, if 
$\alpha C \subset C$ for all $\alpha \in \mathbb{R}_+^*$, and if
$\overline{C} \cap \overline{(-C)}=\{0\}$. An \emph{open cone} in $V$ is a cone which is also an open subset of $V$.
 
Let $C$ be an open cone in $V$. We define a partial order on $V$ by
$$v\leq_C u \ \text{if and only if}\ u-v\in \overline{C}$$
for all $u,v$ in $V$.

\begin{definition}[Asymmetric Thompson distance]
	For $u,v \in C$, define
	\[
	\eta_C(u,v)=\log \inf \{\lambda>0 : v\leq_C \lambda u\}.
	\]
	The function $\eta_C$ is called the {\it asymmetric Thompson distance} on $C$.
\end{definition}

\noindent Observe that $\eta_C$ satisfies the triangle inequality, that it is not symmetric and that it may take negative values. Moreover, for any $\lambda,\mu >0$, we have
$$\eta_C(\lambda u,\mu v)= \eta_C(u,v)+\log \mu-\log \lambda.$$

We need to consider geodesics even for metric spaces whose distance functions may take negative values (but equipped with topologies).
Here is the definition of a \lq\lq generalised'' geodesic.

\begin{definition}[Geodesic]
Let $(X,d)$ be an asymmetric metric space in which $d$ is allowed to take negative values.
We assume that $X$ has a topology with respect which $d$ is continuous.
Then a continuous path $\alpha \colon [a,b] \to X$ is said to be a geodesic if for every
     	\(u,v,w\in [a,b]\) with \(u\leq v\leq w\), we have
     	\[
     	d(\alpha(u),\alpha(w))
     	=
     	d(\alpha(u),\alpha(v))
     	+
     	d(\alpha(v),\alpha(w)).
     	\]
Moreover $\alpha$ is said to be bigeodesic if both $t\mapsto \alpha(t)$ and $t\mapsto \alpha(b-t)$ are geodesics.
\end{definition}
\begin{remark}
\begin{enumerate}
\item 
This definition is the same as the ordinary one if we assume that $d$ takes only non-negative values.
  
  \item
  In this definition,  $\alpha$ is said to be a geodesic between $\alpha(a)$ and $\alpha(b)$. The path $\alpha$ is not necessarily a geodesic between $\alpha(b)$ and $\alpha(a)$.
	Being a bigeodesic means that $\alpha$ is a geodesic between $\alpha(a)$ and $\alpha(b)$, and between $\alpha(b)$ and $\alpha(a)$, at the same time.
     	
 \item With our definition, the fact of being a geodesic is independent of parametrisation.

\item In the following sections, we use symbols such as $\alpha$ and $\alpha'$ to denote a pair of paths in some metric space. In other words, the notation $\alpha'$ does not mean the derivative of $\alpha$.   For the derivative of a path $\alpha$, we shall use the symbol $\dot{\alpha}$. Whenever we use the derivative of a path, it will be understood that this path is piecewise $C^1$.
\end{enumerate}
     \end{remark}

There are several ways of obtaining a symmetric metric on $C$ out of $\eta_C$; one of them is to take the maximum of the two Thompson distances between the two points. It  gives the following:
\begin{definition}[Thompson distance]
The \emph{Thompson distance} $\tau$ is defined by
\[
\tau(u,v)=\max\{\eta_C(u,v),\eta_C(v,u)\}
\]
for any $u, v \in C$.
\end{definition}

\begin{remark}
One may take another symmetrisation of the asymmetric Thompson distance, namely, 
\[
h(u,v)=\frac{1}{2}\bigl(\eta_C(u,v)+\eta_C(v,u)\bigr).
\]
This fonction $h$ 
 satisfies  \[
h(\lambda u,\mu v)=h(u,v)
\]
for all $\lambda,\mu>0$. Thus,  it induces a symmetric metric on the
projectivised cone
\[
\mathbb{P}C=C/\mathbb{R}_+^*.
\]

 The function induced by $h$ on the projectivised cone
 $\mathbb{P}C$ 
 is the well-known \emph{Hilbert metric}. 
 References for this metric include the papers \cite{Gaubert-Gunawardena, Nussbaum1988, Nussbaum1994, Nussbaum-Walsh2004} and the book \cite{Papadopoulos-Troyanov}.
 In the same vein, the asymmetric Thompson distance coincides with the so-called Funk metric after projectivisation; see \cite{Papadopoulos-Troyanov} as well for surveys on the Funk metric.
 \end{remark}

The next lemma says that the definition of the asymmetric Thompson distance is compatible with restriction to vector subspaces.

\begin{lemma}[Restriction to a subspace]
	\label{lemma:subspace}
	Let $W$ be a vector subspace of $V$, and let $C'=W\cap C$. Then, for every
	$u,v\in C'$, we have
	\[
	\eta_{C'}(u,v)=\eta_C(u,v).
	\]
\end{lemma}

\begin{proof}
	Since $u,v$ are in $W$, for any $\lambda >0$, we have $\lambda u-v\in W$. Hence
	\[
	v\leq_{C'} \lambda u
	\]
	if and only if
	\[
	\lambda u-v\in C'=W\cap C,
	\]
	which is equivalent to $\lambda u-v\in C$, that is,
	\[
	u\leq_C \lambda v.
	\]
	Therefore the sets over which the two infima are taken are the same:
	\[
	\{\lambda>0 : v\leq_{C'} \lambda u\}
	=
	\{\lambda>0 : v\leq_C \lambda u\}.
	\]
	Taking logarithms gives
	\[
	\eta_{C'}(u,v)=\eta_C(u,v).
	\]
\end{proof}

Let $K$ be $\overline C$. The dual cone of \(K\) is defined to be
\[
K^*=\{\varphi\in V^*:\ \varphi(x)\geq 0 \text{ for all } x\in K\}.
\]
For \(x,y\in C\), define
\[
M(y/x)=\inf\{\lambda>0:\ y\leq_C \lambda x\}.
\]
We shall use the standard dual characterisation
\[
M(y/x)
=
\sup_{\varphi\in K^*\setminus\{0\}}
\frac{\varphi(y)}{\varphi(x)}.
\]
Since $x$ lies in $C=\operatorname{int}K$, every non-zero
\(\varphi\in K^*\) satisfies \(\varphi(x)>0\). Hence the above quotient  is
well defined. Thus we have
\begin{equation}
\label{dual}
\eta_C(x,y)
=
\log M(y/x)
=
\log
\sup_{\varphi\in K^*\setminus\{0\}}
\frac{\varphi(y)}{\varphi(x)}.
\end{equation}


\begin{lemma}[Geodesics for the asymmetric Thompson distance]
	\label{lemma:standard}
	For \(u,v\in C\), set
	\[
	\gamma(t)=(1-t)u+tv,\qquad 0\leq t\leq 1.
	\]
	Then \(\gamma\) is a bigeodesic for the asymmetric Thompson distance
	\(\eta_C\).
\end{lemma}

\begin{proof}
	Since \(C\) is convex, we have \(\gamma(t)\in C\) for all \(t\in[0,1]\).
	Let
	\[
	\beta=M(v/u).
	\]
	Then \(0<\beta<\infty\). In particular, for every \(a\in[0,1]\), we have
	\[
	1-a+a\beta>0.
	\]
	
	We claim that, for every \(0\leq a\leq b\leq 1\),
	\[
	M(\gamma(b)/\gamma(a))
	=
	\frac{1-b+b\beta}{1-a+a\beta}.
	\]
	This can be shown as follows.
	For \(\varphi\in K^*\setminus\{0\}\), we set
	$\displaystyle
	r_\varphi=\frac{\varphi(v)}{\varphi(u)}.$
	Then \(r_\varphi>0\), and 
	\[
	\frac{\varphi(\gamma(b))}{\varphi(\gamma(a))}
	=
	\frac{(1-b)\varphi(u)+b\varphi(v)}
	{(1-a)\varphi(u)+a\varphi(v)}
	=
	\frac{1-b+br_\varphi}{1-a+ar_\varphi}.
	\]
	For fixed \(0\leq a\leq b\leq1\), define the function
	\[
	F(r)=\frac{1-b+br}{1-a+ar}
	\]
	for $r >0$.
	Then
	\[
	F'(r)=\frac{b-a}{(1-a+ar)^2}\geq0.
	\]
	Thus \(F\) is increasing on \((0,\infty)\).  Since
	$\displaystyle
	\beta=\sup_{\varphi\in K^*\setminus\{0\}}r_\varphi
	$,
	and since \(F\) is increasing and continuous on \((0,\infty)\), we have
	\[
	\sup_{\varphi\in K^*\setminus\{0\}}F(r_\varphi)
	=
	F(\beta).
	\]
	Hence
	\[
	M(\gamma(b)/\gamma(a))
	=
	\frac{1-b+b\beta}{1-a+a\beta}.
	\]
	This proves the claim.
	
	It follows from the claim that
	\[
	\eta_C(\gamma(a),\gamma(b))
	=
	\log M(\gamma(b)/\gamma(a))
	=
	\log(1-b+b\beta)-\log(1-a+a\beta).
	\]
	Therefore, for \(0\leq s\leq t\leq r\leq1\), we have
	\[
	\eta_C(\gamma(s),\gamma(r))
	=
	\log(1-r+r\beta)-\log(1-s+s\beta),
	\]
	whereas
	$$
	\eta_C(\gamma(s),\gamma(t))
	+
	\eta_C(\gamma(t),\gamma(r))
	=
	\bigl[\log(1-t+t\beta)-\log(1-s+s\beta)\bigr]
	+
	\bigl[\log(1-r+r\beta)-\log(1-t+t\beta)\bigr].
	$$
	These two expressions are equal. Hence
	\[
	\eta_C(\gamma(s),\gamma(r))
	=
	\eta_C(\gamma(s),\gamma(t))
	+
	\eta_C(\gamma(t),\gamma(r)).
	\]
	Thus \(\gamma\) is a geodesic for \(\eta_C\).
	Exchanging the roles of $u$ and $v$, we see that $\gamma$ is a bigeodesic.
\end{proof}

\begin{remark}
	\label{remark:crucial}
Suppose that $u$ and $v$ are linearly independent.
	Let $\lambda\colon [0,1]\to \mathbb{R}_+^*$ be a continuous function such that
	$\lambda(0)=\lambda(1)=1$ and let $\gamma$ be as in Lemma \ref{lemma:standard}. Then the path
	\[
	\tilde\gamma(t)=\lambda(t)\gamma(t)
	\]
	is also a geodesic for $\eta_C$. Moreover, $\tilde\gamma$ is a bigeodesic; that is,
	the same path traversed in the reverse sense, $t\mapsto \tilde\gamma(1-t)$ is also a geodesic.
\end{remark}

\subsection{An example: the standard positive cone}

Consider the standard positive cone $C=(\mathbb{R}_+^*)^n$. If
$u=(u_i)$ and $v=(v_i)$ are elements of $C$, then the asymmetric Thompson distance defined in the preceding subsection
 is given by
\[
\eta_C(u,v)=\log \max_{1\leq i\leq n}\frac{v_i}{u_i}.
\]

We now describe the geodesics in $C$.

\begin{proposition}
	\label{main-geodesic}
	A continuous path $\alpha\colon [0,1]\to C$ expressed in coordinates as
	$\alpha(t)=(u_i(t))_{1\leq i\leq n}$ is a geodesic for $\eta_C$ if and only if there exists
	$j\in \{1,\dots,n\}$ such that, for every $t'>t$,
	\[
	\frac{u_i(t')}{u_i(t)}
	\leq
	\frac{u_j(t')}{u_j(t)}
	\qquad \text{for all } i\in \{1,\dots,n\}.
	\]
	Moreover, if $\alpha$ is a geodesic for $\eta_C$, then,
	 for every continuous function
	$\lambda\colon [0,1]\to \mathbb{R}_+^*$, the path
	\[
	\widetilde{\alpha}(t)=\lambda(t)\alpha(t)
	\]
	is also a geodesic for $\eta_C$.
\end{proposition}

\begin{proof}
	Consider the following (asymmetric) Minkowski distance on $\mathbb{R}^n$:
	\[
	\delta(x,y)=\max_{1\leq i\leq n}(y_i-x_i)=\Phi(y-x),
	\]
	where $x=(x_i)$, $y=(y_i)$, and
	\[
	\Phi(z)=\max_{1\leq i\leq n} z_i.
	\]
	The map
	\[
	\operatorname{Log}\colon (\mathbb{R}_+^*)^n\to \mathbb{R}^n,
	\qquad
	(u_i)\mapsto (\log u_i),
	\]
	is an isometry from $(\mathbb{R}_+^*)^n$, equipped with the asymmetric
	Thompson distance $\eta_C$, to $\mathbb{R}^n$, equipped with the  Minkowski distance $\delta$.
	
	We shall say that a continuous path $\beta\colon [0,1]\to \mathbb{R}^n$ is a
	$\delta$-geodesic if
	\[
	\delta(\beta(s),\beta(r))
	=
	\delta(\beta(s),\beta(t))
	+
	\delta(\beta(t),\beta(r))
	\]
	for all $0\leq s\leq t\leq r\leq 1$.
	
	It is easy  to see that a path $\beta(t)=(x_i(t))$ is a
	$\delta$-geodesic if and only if there exists
	$j\in \{1,\dots,n\}$ such that, for every $t'>t$,
	\[
	x_i(t')-x_i(t)
	\leq
	x_j(t')-x_j(t)
	\qquad \text{for all } i.
	\]
	Indeed, this condition says precisely that the same coordinate realises the
	maximum in the expression
	\[
	\delta(\beta(t),\beta(t'))
	=
	\max_i\bigl(x_i(t')-x_i(t)\bigr)
	\]
	for all $t'>t$.
	
Now let $\alpha(t)=(u_i(t))$ and write
	\[
	\operatorname{Log}\alpha(t)=(x_i(t)),
	\qquad
	x_i(t)=\log u_i(t).
	\]
	Since $\operatorname{Log}$ is an isometry, $\alpha$ is a geodesic for
	$\eta_C$ if and only if $\operatorname{Log}\circ \alpha$ is a
	$\delta$-geodesic. Hence $\alpha$ is a geodesic for $\eta_C$ if and only if
	there exists $j\in \{1,\dots,n\}$ such that, for every $t'>t$,
	\[
	\log u_i(t')-\log u_i(t)
	\leq
	\log u_j(t')-\log u_j(t)
	\qquad \text{for all } i.
	\]
	Equivalently,
	\[
	\frac{u_i(t')}{u_i(t)}
	\leq
	\frac{u_j(t')}{u_j(t)}
	\qquad \text{for all } i.
	\]
	
	Finally, if $\beta$ is a $\delta$-geodesic and $b\colon [0,1]\to\mathbb{R}$
	is continuous, then
	\[
	\widetilde{\beta}(t)=\beta(t)+b(t)(1,\dots,1)
	\]
	is again a $\delta$-geodesic, since
	\[
	\delta(\widetilde{\beta}(s),\widetilde{\beta}(t))
	=
	\delta(\beta(s),\beta(t))+b(t)-b(s).
	\]
	Taking $b(t)=\log\lambda(t)$ gives
	\[
	\operatorname{Log}(\lambda(t)\alpha(t))
	=
	\operatorname{Log}\alpha(t)+\log\lambda(t)(1,\dots,1).
	\]
	Therefore $\widetilde{\alpha}(t)=\lambda(t)\alpha(t)$ is also a geodesic for
	$\eta_C$.
\end{proof}

\begin{remark} Geodesics between two points in $C$ are not necessarily unique. 
\end{remark}

\section{The space of Euclidean triangles} \label{s:space}
\begin{definition}[Metric space]
	\label{metric}
	A  metric space is a non-empty set $X$ together with a function 
	$$d: X\times X\to [0,\infty),$$ called a metric,
	satisfying the following two properties:
	
	\begin{enumerate}
		\item 
		$d(x,y)=0$ if and only if $x=y$,
		\item 
		$d(x,y)+d(y,z)\geq d(x,z)$ for all $x,y,z$ in $X$.
	\end{enumerate}
\end{definition}

Thus, a metric, in the sense we are using it in the present paper, does not necessarily satisfy the symmetry axiom, that is, $d(x,y)$ may be different from $d(y,x)$. We shall say that the metric $d$ is \emph{asymmetric} if there exist two points $x$ and $y$ in $X$ such that $d(x,y)\not=d(y,x)$.

We emphasise the fact that given a metric in this sense, there is more than one associated interesting symmetric metric, depending on the way we symmetrise it; this will appear in the following.

In this section, we equip the space of marked Euclidean triangles with a metric. By a \emph{marked} triangle, we mean a triangle with a specification of  the three vertices, in such a way that when we talk about a homeomorphism between two marked triangles, it is understood which vertex of the first one goes to which vertex of the second one. Since all our triangles will be marked, we shall not specify this adjective anymore.

Let \((X,d)\) be a metric space and suppose that it is asymmetric. We associate with \(d\) the following two symmetric metrics on \(X\):
\[
d_{\mathrm{arith}}(x,y)=\frac{1}{2}\bigl(d(x,y)+d(y,x)\bigr)
\]
and
\[
d_{\mathrm{max}}(x,y)=\max\{d(x,y),d(y,x)\}.
\]
These two metrics induce the same topology on \(X\), namely, the one with basis the collection of open balls of the form $B(x,\epsilon)=\{y\in X \ \vert \ d^*(x,y)<\epsilon\}$ where $x$ varies in $X$, $\epsilon$ varies over the set of positive reals, and where $d^*$ could be either $d_{\mathrm{arith}}$ or $d_{\mathrm{max}}$.

 Throughout this paper, whenever \(X\) is regarded as a topological space, it will be endowed with this topology.

\subsection{Distance between triangles}

We parametrise the set of Euclidean triangles by the following subset of $\R^3$:
$$\{(a_1,a_2,a_3) : a_1,a_2,a_3>0, \ a_2+a_3-a_1>0, a_1+a_3-a_2>0, a_1+a_2-a_3>0\}$$
where the positive numbers $a_1, a_2, a_3$ represent the lengths of the  edges of a triangle.

This set can be identified with the product space $(\R^*_+)^3=\{(A_1,A_2,A_3): A_1, A_2, A_3 >0\}$ via the mapping
$$A_1=\frac{a_2+a_3-a_1}{2}$$
$$A_2=\frac{a_3+a_1-a_2}{2}$$
$$A_3=\frac{a_1+a_2-a_3}{2}.$$

The area of a triangle $(a_1,a_2,a_3)$ in terms of $(A_1,A_2,A_3)$ is given by Heron's formula:
\begin{equation}\label{f:Heron}
\mathrm{Area}(a_1,a_2,a_3)=\mathrm{Ar}(A_1,A_2,A_3)=\sqrt{(A_1+A_2+A_3)A_1A_2A_3}.
\end{equation}
We note that the parameters $A_1, A_2, A_3$ are used in the study of circle packings, see e.g.  \cite{Colin}.

We define a function $$\eta: (\R^*_+)^3\times (\R^*_+)^3\to \R$$ by setting
$$\eta((A_1,A_2,A_3),(A_1',A_2',A_3'))=\eta_C((A_1,A_2,A_3),(A_1',A_2',A_3'))=\log \max \{A_1'/A_1 ,A'_2/A_2, A_3'/A_3\},$$
\noindent where $C= (\R_+^*)^3$ is the standard positive cone  $\R^3$, and $\eta_C$ is the asymmetric Thompson distance defined in the preceding section.

The natural action of $\R^*_+$ on $(\R^*_+)^3$ by $\lambda(A_1,A_2,A_3)=(\lambda A_1,\lambda A_2, \lambda A_3)$ corresponds to scaling a triangle by the factor $\lambda$.

For $X$ and $Y$ in $(\R^*_+)^3$ and for $\lambda$ and $\lambda'$ in $\mathbb{R}_+^*$, we have
$$\exp(\eta(\lambda X,\lambda'Y))=\frac{\lambda'}{\lambda}\exp(\eta(X,Y)).$$

Consider the  subspace $\frak{T}_1\subset (\R^*_+)^3$ consisting of unit area triangles:

$$\frak{T}_1=\{(A,B,C): \mathrm{Ar}(A,B,C)=1\}.$$

\begin{proposition}
The function $\eta$ is an asymmetric metric on $\frak{T}_1$. In fact, the metric $\eta$ is the restriction of the Thompson distance  defined on the cone $(\R^*_+)^3$ to the subspace consisting of unit-area triangles.
 \end{proposition}
\begin{proof}
The triangle inequality is obviously satisfied. Assume that $(A,B,C)$ and $(A',B',C')$ are in $\frak{T}_1$ and that $ \eta((A,B,C),(A',B',C'))\leq0$. This means that $A'\leq A$, $B'\leq B$ and $C'\leq C $. Assume, without loss of generality, that $A'<A$.  Then, from Heron's formula recalled in (\ref{f:Heron}), we have 
$$\mathrm{Ar}(A',B',C')<\mathrm{Ar}(A,B,C),$$
which is a contradiction. Therefore $A=A'$. Thus, $\eta((A,B,C), (A',B',C'))\leq 0$ implies $A=A'$, $B=B'$ and $C=C'$. Asymmetry may be checked by looking e.g.\ at the distances between the points $(1,1,1)$ and $(\sqrt{3}/2, \sqrt{3}/2, 1-\sqrt{3}/2)$.
\end{proof}

%
%
     
%
%
%
%

\begin{theorem}[Characterisation of geodesics in $ \frak{T}_1$]
\label{geodesic}
Let $\alpha: [0,1]\to \frak{T}_1$ be a piecewise $C^1$ path, $\alpha(t)=(A_1(t),A_2(t),A_3(t))$.
Then $\alpha$ is a geodesic in $\frak{T}_1$ (with respect to the metric $\eta$) if and only if there exists some $j \in \{1,2,3\}$ such that for each $t, t' \in [0,1]$ with $t\leq t'$  we have
$$\frac{A_i(t')}{A_i(t)}\leq  \frac{A_j(t')}{A_j(t)}\ \text{for all}\ i\in \{1,2,3\}.$$

\end{theorem}
\begin{proof}
This follows from Proposition \ref{main-geodesic}. Indeed a path in $\frak{T}_1$ is a geodesic if and only if it is a geodesic with respect to Thompson's asymmetric distance $\eta_C$, where $C=(\R^*_+)^3$.
\end{proof}

\begin{corollary}
	\label{corrollary-log-derivative}
Let $\alpha: [0,1]\to \frak{T}_1$ be a $C^1$ path with $\alpha(t)=(A_1(t), A_2(t), A_3(t))$. Suppose that  there exists $j \in \{1,2,3\}$ such that for each $i\in \{1,2,3\}$ and  for each $s \in [0,1]$, we have
$$\frac{\dot{A}_i(s)}{A_i(s)}\leq \frac{\dot{A}_j(s)}{A_j(s)}.$$
Then $\alpha(t)$ is a geodesic in $\frak{T}_1$.
\end{corollary}
\begin{proof}
Under the hypothesis, it is clear that 
$$\log A_i(t')-\log A_i(t)\leq \log A_j(t')-\log A_j(t)\ \text{for all}\ i\in \{1,2,3\} \ \text{and for all}\ t\leq t'.$$
\end{proof}

\begin{corollary}
	\label{cor:bigeodesic}
	For any two distinct points in $\mathfrak{T}_1$, there exists a bigeodesic
	joining them.
\end{corollary}

\begin{proof}
	Let $(A,B,C)$ and $(A',B',C')$ be two distinct points in $\mathfrak{T}_1$.
	Define a path $\alpha\colon [0,1]\to (\mathbb{R}_+^*)^3$ by
	\[
	\alpha(t)=\bigl(A^{1-t}(A')^t,\; B^{1-t}(B')^t,\; C^{1-t}(C')^t\bigr).
	\]

	\noindent 	For each $t\in [0,1]$, let $\lambda(t)>0$ be the unique real number such that the path
	\[
	\widetilde{\alpha}(t)=\lambda(t)\alpha(t)
	\]
	belongs to $\mathfrak{T}_1$. Corollary \ref{corrollary-log-derivative} implies that $\widetilde{\alpha}$ is a geodesic.
	
	The same argument applied to the reverse path shows that
	$t\mapsto \widetilde{\alpha}(1-t)$ is also a geodesic. Hence
	$\widetilde{\alpha}$ is a bigeodesic.
\end{proof}

\subsection{The Finsler structure of $\frak{T}_1$}
\label{finsler}
In this subsection, we show that the metric $\eta$ on $\frak{T}_1$ is Finsler, and give a formula for its infinitesimal Finsler norm. Before that, we need to show that $\frak{T}_1$ is a differentiable submanifold of $(\R_+^*)^3$. We recall that
$$\frak{T}_1=\{(A_1,A_2,A_3) : A_i>0, (A_1+A_2+A_3)A_1A_2A_3=1\}.$$ 
 Let $q:(\R^*_+)^3\to \R^*_+$ be the map defined by $$(A,B,C)\mapsto (A+B+C)ABC.$$ 
	Clearly, this map is differentiable.
	The differential  is everywhere non-zero (see the proof of Lemma \ref{non-zero}), and in particular $1$ is a regular value. Therefore $\frak{T}_1$ is an embedded differentiable submanifold of $(\R^*_+)^3$.

We recall the definition of a Finsler structure, in this setting of non-necessarily symmetric metric spaces. We start with the definition of a Minkowski norm. 

\begin{definition}[Minkowski norm]
\label{Minkowski norm}
	Let $V$ be a real vector space. A \emph{Minkowski norm} 
	on $V$ is a function $V\to [0,\infty)$, $v \mapsto \lvert\lvert v \rvert \rvert$ such that
	 the following properties hold for every $v$ and $w$ in $V$:
	\begin{enumerate}
		\item 
		$\lvert\lvert v\rvert\rvert=0$ if and only if $v$ is the zero vector in $V$;
		\item
		$\lvert\lvert t v\rvert\rvert= t \lvert\lvert v \rvert \rvert $ for every $t>0$;
		\item
		$\lvert\lvert tv+(1-t)w\rvert\rvert\leq t\lvert\lvert v \rvert\rvert + (1-t)\lvert\lvert w \rvert\rvert$ for every $t\in [0,1]$.
	\end{enumerate}
\end{definition}

%
%

The epithet \emph{Minkowski} is used in Definition \ref{Minkowski norm} in order to distinguish this notion from that of \emph{norm} in which Condition (2) is replaced by the following stronger condition: 
$\lvert\lvert t v\rvert\rvert= t \lvert\lvert v \rvert \rvert$ for every $t\in \mathbb{R}$.

The term Minkowski norm, with the same definition, is used in Finsler geometry, see e.g. \cite{MT}.
Thus, it is understood that our norms may not satisfy the symmetry axiom.

Let $M$ be a differentiable manifold and  $TM$ its tangent bundle. 

\begin{definition}[Finsler structure]
	A Finsler structure on $M$ is a function $F: TM \to [0,\infty)$ such that
	
	\begin{enumerate}
		\item 
		$F$ is continuous;
		\item
		for each $x\in M$, $F\lvert_{T_xM}$ is a Minkowski norm.
	\end{enumerate}
\end{definition}

Let $F$ be a Finsler structure on a manifold $M$. For each piecewise $C^1$ curve $c: [a,b]\to M$, its length  is defined by  
$$l(c)=l_F(c)=\int_{a}^{b}F(\dot{c}(t))dt.$$

The following notion of Finsler metric is a classical one, see e.g. \cite{MT}. 
\begin{definition}[Finsler metric]
	A metric $d$ on a differentiable manifold $M$ is said to be \emph{Finsler} if it is the length metric associated with a Finsler structure, that is, if there exists a Finsler structure $F$
	on $M$ such that for every $x,y \in M$ we have
	$$d(x,y)= \inf\{l_F(c)\}$$
	 where $c$ ranges over all piecewise $C^1$ curves such that $c(0)=x$ and $c(1)=y$. 
\end{definition}

Now we start the proof of the fact that the metric $\eta$ on $\frak{T}_1$ is Finsler. 
We shall give the precise expression of the associated family of norms on tangent spaces.  

 We consider the following function $F$ on the tangent bundle $T(\frak{T}_1)$ of $\frak{T}_1$:
$$F: (A_1,A_2, A_3,v_1, v_2,v_3)\mapsto \max_{i}\{\frac{v_i}{A_i}\}.$$
Here, $A=(A_1,A_2,A_3)$ is a point in
$\frak{T}_1$ and $(v_1, v_2, v_3)$ are the coordinates of a tangent vector at $A=(A_1,A_2, A_3)$. 
We need to prove that $F\lvert_{T_A\frak{T}_1}$ is non-negative and satisfies Condition (1) of Definition \ref{Minkowski norm}. 
The other conditions of a Minkowski norm are clearly satisfied. 

\sloppy
\begin{lemma}
\label{non-zero}
The function $F\lvert_{T_A\frak{T}_1}$ takes only non-negative values. Moreover, $F\lvert_{T_A\frak{T}_1}(A_1, A_2, A_3, v_1, v_2, v_3)=0$ if and only if $(v_1, v_2, v_3)=(0,0,0)$.
\end{lemma}
\fussy
\begin{proof}
By taking the differential of the area formula $A_1A_2A_3(A_1+A_2+A_3)=1$, we have 
\begin{equation}
\label{diff area}
\begin{split}
&(2A_1A_2A_3+A_2A_3(A_2+A_3))v_1+(2A_1A_2A_3+A_1A_3(A_1+A_3))v_2\\&+(2A_1A_2A_3+A_1A_2(A_1+A_2))v_3=0.\end{split}\end{equation}
Since $A_1, A_2$ and $A_3$ are all positive, the equality cannot hold if all of $v_1,v_2, v_3$ are negative.
Therefore $\max_{i}\{\frac{v_i}{A_i}\}$ is non-negative.
If $\max_{i}\{\frac{v_i}{A_i}\}=0$, then none of $v_1, v_2, v_3$  are positive, and the equality \ref{diff area} shows that they must be all equal to $0$.
\end{proof}
%
%

\begin{proposition}
	\label{curves}
	Let  $A'=(A'_1,A'_2,A'_3)$ and $A''=(A''_1, A''_2,A''_3)$ be two points in $\frak{T}_1$. Then
	$$\eta(A',A'')=\inf_c\{l_F(c)\}$$
	where $c$ ranges over all piecewise $C^1$ curves contained in $\frak{T}_1$ such that $c(0)=A'$, $c(1)=A''$
	and where
$$l_F(c)=\int_{0}^1 F(\dot{c}(t)) dt.$$
	\end{proposition}
	
	\begin{proof}
	By the usual approximation techniques, we only have to consider the case when $c$ is a $C^1$ curve.
		Let $c(t)=(A_1(t),A_2(t),A_3(t))$ be a $C^1$ curve in $\frak{T}_1$ such that $c(0)=A'$ and $c(1)=A''$. Then 
		$$l_F(c)=\int_0^1\max_i\{\frac{\dot{A}_i(t)}{A_i(t)}\}\ dt\geq \max_i\{ \int_0^1\frac{\dot{A}_i(t)}{A_i(t)}\ dt\}= \max_i \{\log(\frac{A_i''}{A_i'})\}=\eta(A',A'').$$
		
	Now we prove that $l_F(c)=\eta(A',A'')$ for the geodesics $c$ which we defined in Lemma \ref{cor:bigeodesic}.  Assume that 
	$$\max_i\{\frac{A_i''}{A_i'}\}=\frac{A_j''}{A_j'}.$$
	Let $c:[0,1]\to \frak{T}_1$ be a curve expressed by $c(t)=(A_1(t),A_2(t), A_3(t))$ such that
	$$\frac{\dot{A}_i(t)}{A_i(t)}\leq\frac{\dot{A}_j(t)}{A_j(t)}.$$
Then, 
$$l_F(c)=\int_0^1\max_i\{ \frac{\frac{d}{dt}(A_i(t))}{A_i(t)} \}\ dt=\int_0^1 \frac{\frac{d}{dt}(A_j(t))}{A_j(t)}\ dt$$
$$ = \log A''_j -\log A'_j =\eta(A',A'').$$
\end{proof}

%
%

The following theorem is an immediate consequence of Proposition \ref{curves}.

\begin{theorem}
Let 	$F$ be the following function  on the tangent space 
	$T(\frak{T}_1)$ of $\frak{T}_1$:
	$$ (A_1, A_2, A_3 ,v_1,v_2,v_3)\mapsto \max_{i}\{\frac{v_i}{A_i}\}.$$
	The metric $\eta$  is induced
	by the  Finsler structure on $\frak{T}_1$
	given by $F$.
\end{theorem}

\section{The case of a  Euclidean structure adapted to a fixed triangulation of a surface} \label{s:surface}
Let $S$ be a compact surface possibly with boundary and $\mathcal{T}$  a finite  triangulation of $S$.
We denote the set of edges and vertices of $\mathcal{T}$ by
 $E(\mathcal{T})$ and 
 $V(\mathcal{T})$, respectively.

Let $\mathcal{E}(S,\mathcal{T})$ be the set of singular Euclidean metrics adapted to the triangulation $\mathcal{T}$. This means that the metric is locally Euclidean except perhaps at the vertices. Note that such a structure is determined by the collection of edge lengths of the triangulation, which can be chosen arbitrarily subject to the triangle inequalities for each face. We shall call such a metric a Euclidean structure on $(S,\mathcal{T})$.

Let $\mu$ be a Euclidean structure on $(S,\mathcal{T})$. There is a naturally induced notion of length of a piecewise $C^1$ curve on $S$. 
We
 denote by $l_{\mu}(c)$ the length of such a curve $c$. This defines a length structure $l_{\mu}$ on $S$ which induces in the usual way a metric in $S$ which we denote by $d_{\mu}$. If $M$ is a Lebesgue measurable subset of $S$, we denote its area 
 by $A_{\mu}(M)$.

 The map $\Phi: \mathcal{E}(S,\mathcal{T})\to (\R_{+}^*)^{E(\mathcal{T})}$ sending an element $\mu$ to the sequence $(l_{\mu}(e_i))$, where $e_i\in E(\mathcal{T})$, is injective. The image of this map is  the set of all $(a_i)\in (\R_{+}^*)^{E(\mathcal{T})}$
 such that whenever $e_i$, $e_j$ and $e_k$ are assigned to a face of  $\mathcal{T}$, the associated triple $a_i,a_j$ and $a_k$ satisfy the triangle inequality.

Let  $\mathcal{J}$ be a subset of $E(\mathcal{T})^3$ with the following properties. 

\begin{enumerate}
\item
For each $(i,j,k)\in \mathcal{J}$, $e_i$, $e_j$ and $e_k$ are edges of 
a face of  $\mathcal{T}$. 
\item
If $f$ is a face of $\mathcal{T}$ with edges $e_i$, $e_j$, and $e_k$,  then some permutation of $(i,j,k)$ is in $\mathcal{J}$.
\item
If $(i,j,k)$ is in $\mathcal{J}$, then no odd  permutation of $(i,j,k)$ is in $\mathcal{J}$.
\item
If $(i,j,k)$ is in $\mathcal{J}$, then all even permutations of $(i,j,k)$ are in $\mathcal{J}$.
 \end{enumerate}

It is quite easy to see the following.
\begin{proposition}
The map $\Psi: \mathcal{E}(S,\mathcal{T})\to (\R_+^*)^{\mathcal{J}}$ sending $[\mu]$ to $(A_{ijk})$, defined by  $$\displaystyle A_{ijk}=\frac{l_{\mu}(e_j)+l_{\mu}(e_k)-l_{\mu}(e_i)}{2},$$ is injective.
\end{proposition}
Assume that there are two elements $(i,j,k)$ and $(i',j',k')$ in $\mathcal{J}$ such that $i=i'$, $j\neq j'$ and $k\neq k'$. Take a point $A$ in the image of $\Psi$. Then we have the equality
$$A_{jki}+A_{kij}=A_{j'k'i'}+A_{k'i'j'},$$ corresponding to the interior edge $e_i$.
Conversely, this is the only constraint  for points in $(\R_+^*)^{\mathcal{J}}$ to be realised as edge lengths of $\mathcal T$.
It follows that the image of $\Psi$ is a subset of $(\R_+^*)^{\mathcal{J}}$ which is an intersection of hyperplanes corresponding to interior edges of the triangulation $\mathcal{T}$. 
We denote this set by $\mathcal{E}(S,\mathcal{T})$ as well.

  Let $A$ be an element of $\mathcal{E}(S,\mathcal{T})$. We define  $\mathrm{Ar}(A)$ to be the area of corresponding singular flat metric on $S$.
We denote the subset of $\mathcal{E}(S,\mathcal{T})$ which consists of elements $A$ with $\mathrm{Ar}(A)=1$ by $\mathcal{E}(S,\mathcal{T})_1$. Cleary, there is an action
of $\R_+^*$ on $\mathcal{E}(S,\mathcal{T})$ corresponding to scaling, and we have
$$\mathrm{Ar}(\lambda A)=\lambda^2 \mathrm{Ar}(A),\ \lambda\in \R_+^*.$$

 \begin{definition}
 Let $A=(A_{ijk})$ and $ B=(B_{ijk})$ be two elements of $ \mathcal{E}(S,\mathcal{T})$. We define 
 $$\eta(A,B):= \eta_C(A,B)=\log \max_{(i,j,k)\in \mathcal{J}}\frac{B_{ijk}}{A_{ijk}},$$
 \noindent where $C$ is the cone $(\R_+^*)^{\mathcal{J}}$.  
 \end{definition}

Observe that $\exp(\eta(\lambda A,\lambda' B))=\frac{\lambda'}{\lambda}\exp((A,B))$.
 The function $\eta$ is a version of the generalised  metric  studied in Section \ref{s:space} and it corresponds to the asymmetric Thompson distance we defined in \S  \ref{s:Thompson}.
 
 \begin{proposition}
 $\eta$ is a metric on $\mathcal{E}(S,\mathcal{T})_1$.
  \end{proposition}
 \begin{proof}
 The triangle inequality is obvious. Suppose that for $A, B \in \mathcal{E}(S,\mathcal{T})_1$, we have $\eta(A,B)\leq 0$. Then, $A_{ijk}\leq B_{ijk}$ for each $(i,j,k)\in \mathcal{J}$. If one of these inequalities is strict, 
 then $\mathrm{Ar}(A)<\mathrm{Ar}(B)$, which is a contradiction. Therefore $\eta(A,B)\leq 0$ implies $A=B$. 
 \end{proof}
 
\begin{theorem}
	\label{cor:bigeodesic-2}
	For any two points $u$ and $v$ of $\mathcal{E}(S,\mathcal{T})_1$, there exists a
	bigeodesic joining them with respect to $\eta$, which is expressed as $\lambda(t)((1-t)u+tv)$.
\end{theorem}

\begin{proof}
	Consider the map
	\[
	\Psi\colon \mathcal{E}(S,\mathcal{T})\to (\mathbb{R}_+^*)^{\mathcal{J}}.
	\]
	Let
	\[
	V=\operatorname{span}\bigl(\Psi(\mathcal{E}(S,\mathcal{T}))\bigr)
	\subset \mathbb{R}^{\mathcal{J}}.
	\]
	Then
	\[
	C'=\Psi(\mathcal{E}(S,\mathcal{T}))
	=
	V\cap (\mathbb{R}_+^*)^{\mathcal{J}}
	\]
	is an open cone in $V$. Since $\Psi$ is injective,
	we may identify $\mathcal{E}(S,\mathcal{T})$ with $C'$. Let us denote the cone
	$(\mathbb{R}_+^*)^{\mathcal{J}}$ by $C$.
	
	Moreover, under this identification, by Lemma \ref{lemma:subspace}, the
	restriction of $\eta_C$ to $C'$ is precisely the asymmetric Thompson distance
	$\eta_{C'}$ on the cone $C'$. Therefore, Lemma \ref{lemma:standard} implies
	that for any two points $u,v \in \mathcal{E}(S,\mathcal{T})_1$, the line
	segment
	\[
	\gamma(t)=(1-t)u+tv
	\]
	is a bigeodesic with respect to the asymmetric Thompson distance $\eta_C$
	on $C$. It follows from Remark \ref{remark:crucial} that the normalised
	affine segment
	\[
	\tilde{\gamma}(t)=\lambda(t)\gamma(t)
	=
	\lambda(t)\bigl((1-t)u+tv\bigr),
	\]
	where $\lambda(t)>0$ is chosen so that
	$\tilde{\gamma}(t)\in \mathcal{E}(S,\mathcal{T})_1$, is a bigeodesic for
	$\eta_C$. This implies the desired result since the metric $\eta$ is the
	restriction of the asymmetric Thompson distance on $C$ to
	$\mathcal{E}(S,\mathcal{T})_1$.
\end{proof}

  \begin{theorem}
 	\label{geodesic1}
 	Let $\alpha: [0,1]\to \mathcal{E}(S,\mathcal{T})_1$ be a continuous  path in $\mathcal{E}(S,\mathcal{T})_1$, and set $\alpha(t)=(A_{ijk}(t))$. Then $\alpha$ is a geodesic with respect to the metric $\eta$ if and only if there exists a triple  $(i',j',k')\in \mathcal{J}$ such that for each $t, t' \in [0,1]$ with $t\leq t'$  we have
 	$$ \frac{A_{ijk}(t')}{A_{ijk}(t)}\leq  \frac{A_{i'j'k'}(t')}{A_{i'j'k'}(t)}\ \text{for all}\ (i,j,k)\in \mathcal{J}.$$
 \end{theorem}
 \begin{proof}
 	This follows from Proposition \ref{main-geodesic}. Indeed a path in $\mathcal{E}(S,\mathcal{T})_1$ is a geodesic if and only if it is a geodesic with respect to the asymmetric Thompson distance $\eta_C$, where $C=(\R^*_+)^\mathcal{J}$.
 \end{proof}

\subsection{Finsler structure on $\mathcal{E}(S,\mathcal{T})_1$}

We start with the following.

\begin{proposition}
	$\mathcal{E}(S,\mathcal{T})_1$ is an embedded differentiable submanifold of $(\R_+^*)^{\mathcal{J}}$.
\end{proposition}
	\begin{proof}
	The argument is the same as the one we used for the space  $\frak{T}_1$: 
	Let $q: \mathcal{E}(S,\mathcal{T})\to \R^*_+$ be the area map  $A\mapsto \mathrm{Ar}(A)$. Then clearly this map is differentiable and $q(\lambda A)=\lambda^2q(A)$
	implies that either all points $\R^*_+$ are critical values of $q$, or none of them are. 
	By Sard's theorem, the first case is impossible, hence none of the points of $\R^*_+$ are critical values. It follows then that $1$ is a regular value, and $\mathcal{E}(S,\mathcal{T})_1$ is an embedded differentiable submanifold of $\mathcal{E}(S,\mathcal{T})$. Since  $\mathcal{E}(S,\mathcal{T})$ is an intersection of hyperplanes, the result follows.
	\end{proof}

Consider the following function defined on the tangent bundle $T(\mathcal{E}(S,\mathcal{T})_1)$:
$$F({\mathcal{T}}): ((A_{ijk}),(v_{ijk}))\mapsto \max_{(i,j,k)\in \mathcal J}\{\frac{v_{ijk}}{A_{ijk}}\},$$
where $(A_{ijk})$ lies in $\mathcal{E}(S,\mathcal{T})_1$ and $v_{ijk}$  are the coordinates of tangent vectors at the point $(A_{ijk})$.

\begin{proposition}
	Let  $A=(A_{ijk})$ and $B=(B_{ijk})$ be in $\mathcal{E}(S,\mathcal{T})_1$. Then
	$$\eta(A,B)=\inf\{l_F(c)\}$$
	\noindent where $c$ ranges over all piecewise $C^1$ curves in $\mathcal{E}(S,\mathcal{T})_1$ such that $c(0)=A$, $c(1)=B$.

\end{proposition}
\begin{proof}
	The proof is similar to that of Proposition~\ref{curves}.
	Let \(c(t)=(A_{ijk}(t))\) be a \(C^1\)-curve in
	\(\mathcal{E}(S,\mathcal{T}_1)\) such that
	\(c(0)=A=(A_{ijk})\) and \(c(1)=B=(B_{ijk})\). Then
	\[
	\begin{aligned}
		l_F(c)
		&=
		\int_0^1
		\max_{(i,j,k)\in \mathcal{J}}
		\left\{
		\frac{\dot A_{ijk}(t)}{A_{ijk}(t)}
		\right\}\,dt  \\
		&\geq
		\max_{(i,j,k)\in \mathcal{J}}
		\left\{
		\int_0^1
		\frac{\dot A_{ijk}(t)}{A_{ijk}(t)}\,dt
		\right\}  \\
		&=
		\max_{(i,j,k)\in \mathcal{J}}
		\left\{
		\log\frac{B_{ijk}}{A_{ijk}}
		\right\}
		=
		\eta(A,B).
	\end{aligned}
	\]
	
	We now prove that \(l_F(\gamma)=\eta(A,B)\) for the geodesic
	\(\gamma\) constructed in the proof of Theorem~\ref{cor:bigeodesic-2}.
	Let \(\gamma(t)=(A_{ijk}(t))\), with
	\[
	(A_{ijk}(0))=(A_{ijk})
	\qquad \text{and} \qquad
	(A_{ijk}(1))=(B_{ijk}).
	\]
	The curve \(\gamma\) is of class \(C^1\). By Theorem~\ref{geodesic1},
	there exists \((i',j',k')\in \mathcal{J}\) such that for all $0\leq t \leq t' \leq 1$ and $(i,j,k) \in \mathcal J$, we have $$\log A_{ijk}(t')-\log A_{ijk}(t)\leq \log A_{i'j'k'}(t')-\log A_{i'j'k'}(t).$$
	Hence for all \(t\in[0,1]\),
	\[
	\frac{\dot A_{ijk}(t)}{A_{ijk}(t)}
	\leq
	\frac{\dot A_{i'j'k'}(t)}{A_{i'j'k'}(t)}
	\]
	for every \((i,j,k)\in \mathcal{J}\), and
	\[
	\max_{(i,j,k)\in \mathcal{J}}
	\left\{
	\frac{B_{ijk}}{A_{ijk}}
	\right\}
	=
	\frac{B_{i'j'k'}}{A_{i'j'k'}}.
	\]
	It follows that
	\[
	\begin{aligned}
		l_F(\gamma)
		&=
		\int_0^1
		\max_{(i,j,k)\in \mathcal{J}}
		\left\{
		\frac{\dot A_{ijk}(t)}{A_{ijk}(t)}
		\right\}\,dt  \\
		&=
		\int_0^1
		\frac{\dot A_{i'j'k'}(t)}{A_{i'j'k'}(t)}\,dt  \\
		&=
		\log B_{i'j'k'}-\log A_{i'j'k'}  \\
		&=
		\log\frac{B_{i'j'k'}}{A_{i'j'k'}}
		=
		\eta(A,B).
	\end{aligned}
	\]
	This completes the proof.
\end{proof}
\begin{proposition}
	 For each $A\in \mathcal{E}(S,\mathcal{T})_1$, $F({\mathcal{T}})\lvert_{T_A(\mathcal{E}(S,\mathcal{T})_1)}$ is a Minkowski norm.
	\end{proposition}
	\begin{proof}
	Since the sum of the areas of all faces of $\mathcal T$ is $1$, by taking the derivative and repeating the argument of the proof of Lemma \ref{non-zero}, we can show that $F({\mathcal{T}})\lvert_{T_A(\mathcal{E}(S,\mathcal{T})_1)}$ is non-negative and takes positive values except for the zero vector.
	\end{proof}

The previous two propositions easily imply the following.

\begin{theorem}
Let 	$F({\mathcal{T}})$ be the following function  on the tangent bundle
	$T(\mathcal{E}(S,\mathcal{T})_1)$ of $\mathcal{E}(S,\mathcal{T})_1$:
	$$ ((A_{ijk}),(v_{ijk}))\mapsto \max_{(ijk)\in \mathcal J}\{\frac{v_{ijk}}{A_{ijk}}\}.$$
 The metric $\eta$  is induced
	by the  Finsler structure on $\mathcal{E}(S,\mathcal{T})_1$
	given by $F({\mathcal{T}})$.
\end{theorem}
\subsection{The Finsler structure of the asymmetric Thompson distance}
We now turn to the general setting of the asymmetric Thompson distance.
We shall show that this distance on an open  cone in $\R^n$ induces a Finsler structure on a submanifold of $\R^n$ whose tangent spaces are disjoint from the cone.
Moreover, in the case when every ray in the cone intersects  $N$ at a unique point, it will be shown that the  Finsler metric coincides with the original asymmetric Thompson metric.
This will give a general perspective of what we have proved in the preceding subsection.

	Let \(V\) be a finite-dimensional real vector space, and let
\(C\subset V\) be an open convex cone. Set \(K=\overline C\).

\begin{lemma}

	For every \(x\in C\) and \(\xi\in V\), we have
	\[
	\lim_{t\to 0^+}
	\frac{\eta_C(x,x+t\xi)}{t}
	=
	\sup_{\varphi\in K^*\setminus\{0\}}
	\frac{\varphi(\xi)}{\varphi(x)}.
	\]
\end{lemma}

\begin{proof}
	Since \(C\) is open, \(x+t\xi\in C\) for all sufficiently small
	\(t>0\). Using (\ref{dual}), we have
	\[
	\eta_C(x,y)
	=
	\log
	\sup_{\varphi\in K^*\setminus\{0\}}
	\frac{\varphi(y)}{\varphi(x)},
	\]
	and we get
	\[
	\eta_C(x,x+t\xi)
	=
	\log
	\sup_{\varphi\in K^*\setminus\{0\}}
	\frac{\varphi(x+t\xi)}{\varphi(x)}.
	\]
	Since $\varphi$ is linear,
	\[
	\frac{\varphi(x+t\xi)}{\varphi(x)}
	=
	1+t\frac{\varphi(\xi)}{\varphi(x)}.
	\]
	Therefore
	\[
	\eta_C(x,x+t\xi)
	=
	\log\left(
	1+t
	\sup_{\varphi\in K^*\setminus\{0\}}
	\frac{\varphi(\xi)}{\varphi(x)}
	\right).
	\]
	Setting
	\[
	A=
	\sup_{\varphi\in K^*\setminus\{0\}}
	\frac{\varphi(\xi)}{\varphi(x)},
	\]
	we obtain
	\[
	\frac{\eta_C(x,x+t\xi)}{t}
	=
	\frac{\log(1+tA)}{t}.
	\]
	Letting \(t\to 0^+\), we have
	\[
	\lim_{t\to 0^+}
	\frac{\eta_C(x,x+t\xi)}{t}
	=A,
	\]
	as required.
\end{proof}

\begin{remark}

	Consider the tangent bundle $TC$ identified to $C\times V$. Define
	\(F:TC\to\mathbb R\) by
	\[
	F(x,\xi)
	=
	\sup_{\varphi\in K^*\setminus\{0\}}
	\frac{\varphi(\xi)}{\varphi(x)}.
	\]
	Then:
	\begin{enumerate}
		\item \(F\) is continuous;
		\item \(F(x,t\xi)=tF(x,\xi)\) for every \(t>0\);
		\item for every \(s\in[0,1]\),
		$
		F(x,s\xi+(1-s)\xi')
		\le
		sF(x,\xi)+(1-s)F(x,\xi').
		$
	\end{enumerate}
\end{remark}
\begin{lemma}
	Let \(N\) be a submanifold of \(C\) such that
	\[
	T_xN\cap K=\{0\}
	\]
	for every \(x\in N\). Then \(F|_{TN}\) defines a Finsler structure on $N$.
	\end{lemma}

\begin{proof}
	By the previous remark, it only remains to prove positivity on non-zero
	tangent vectors. Let $x$ be a point on $N$, and $\xi$ a non-zero tangent vector at $x$.
	 Suppose, seeking a contradiction, that
	\[
	F(x,\xi)\le 0.
	\]
	Then
	\[
	\sup_{\varphi\in K^*\setminus\{0\}}
	\frac{\varphi(\xi)}{\varphi(x)}
	\le 0.
	\]
	Since $x$ lies in $C$, we have \(\varphi(x)>0\) for every
	nonzero dual vector $\varphi\in K^*$. 
	Hence
	\[
	\varphi(\xi)\le 0
	\qquad
	\text{for all } \varphi\in K^*.
	\]
	By the dual characterisation of \(K\) by $K^*$, this implies that
	\[
	-\xi\in K.
	\]
	Therefore,
	\[
	-\xi\in T_xN\cap K=\{0\},
	\]
	which gives \(\xi=0\), a contradiction. Hence \(F(x,\xi)>0\) for every
	\(0\neq \xi\in T_xN\). Thus we have shown that \(F|_{TN}\) is a Finsler structure.
\end{proof}

\begin{theorem}
	Let \(N\) be a submanifold of \(C\) such that
	\[
	T_xN\cap K=\{0\}
	\]
	for every \(x\in N\). Assume that every ray in \(C\) intersects
	\(N\) in exactly one point. Then the Finsler metric  induced by
	\(F|_{TN}\) coincides with the restriction of \(\eta_C\) to \(N\). To be more precise, for every \(x,y\in N\), we have
	\[
	\eta_C(x,y)
	=
	\inf_c \int_0^1 F(c(t),\dot c(t))\,dt,
	\]
	where the infimum is taken over all piecewise \(C^1\) curves
	\(c:[0,1]\to N\) such that \(c(0)=x\) and \(c(1)=y\).
\end{theorem}

\begin{proof}
	Let \(c:[0,1]\to N\) be a piecewise \(C^1\) curve joining \(x\) to \(y\).
	For every \(\varphi\in K^*\setminus\{0\}\),
	\[
	\frac{d}{dt}\log\varphi(c(t))
	=
	\frac{\varphi(\dot c(t))}{\varphi(c(t))}
	\le F(c(t),\dot c(t)).
	\]
	Integrating this, we have 
	\[
	\log\frac{\varphi(y)}{\varphi(x)}
	\le
	\int_0^1 F(c(t),\dot c(t))\,dt .
	\]
	Taking the supremum over \(\varphi\), we get
	\[
	\eta_C(x,y)
	\le
	\int_0^1 F(c(t),\dot c(t))\,dt .
	\]
	
	To obtain the opposite inequality, let  $z(t)$ be the segment 
	\[
	z(t)=(1-t)x+ty.
	\]
	By assumption, there is a positive function \(\lambda\), with
	\(\lambda(0)=\lambda(1)=1\), such that
	\[
	c(t)=\lambda(t)z(t)\in N.
	\]
	Set
	\[
	A=
	\sup_{\varphi\in K^*\setminus\{0\}}
	\frac{\varphi(y)}{\varphi(x)}.
	\]
	Then \(\eta_C(x,y)=\log A\). Since
	\[
	\dot c(t)=\dot\lambda(t)z(t)+\lambda(t)(y-x),
	\]
	we have
	\[
	F(c(t),\dot c(t))
	=
	\frac{\dot\lambda(t)}{\lambda(t)}
	+
	\sup_{\varphi\in K^*\setminus\{0\}}
	\frac{\varphi(y-x)}{\varphi(z(t))}.
	\]
	Setting \(r_\varphi=\varphi(y)/\varphi(x)\), we can write
	\[
	\frac{\varphi(y-x)}{\varphi(z(t))}
	=
	\frac{r_\varphi-1}{(1-t)+tr_\varphi}.
	\]
	Since the last expression is increasing in \(r_\varphi>0\),
	\[
	F(c(t),\dot c(t))
	=
	\frac{\dot\lambda(t)}{\lambda(t)}
	+
	\frac{A-1}{(1-t)+tA}.
	\]
	Therefore
	\[
	\int_0^1 F(c(t),\dot c(t))\,dt
	=
	\log\lambda(1)-\log\lambda(0)
	+
	\int_0^1\frac{A-1}{(1-t)+tA}\,dt
	=
	\log A.
	\]
	Hence this curve has length equal to \(\eta_C(x,y)\), and we have the opposite inequality.
\end{proof}

\section{Some questions regarding families of weighted metrics}\label{s:weighted}
We consider the family of metrics on $\frak{T}_1$, parametrised by $t\in [0,1]$ and defined for each such $t$ by the formula
$$\eta^a_t(X,Y)=(1-t)\eta(X,Y)+t\eta(Y,X),$$
for $X,Y \in \frak{T}_1$. Here, the superscript $a$ in $\eta^a_t$ stands for ``arithmetic". 

Let 	$F^a _t$ be the following function  on the tangent space 
$T(\frak{T}_1)$ of $\frak{T}_1$:

$$ (A_1, A_2, A_3 ,v_1,v_2,v_3)\mapsto [(1-t)\max_{i}\{\frac{v_i}{A_i}\}+t \max_{i}\{\frac{-v_i}{A_i}\}].$$

We have the following:

\begin{theorem}   For every $t\in [0,1]$, the metric $\eta^a_t$  is Finsler, and its infinitesimal Finsler structure on $\frak{T}_1$ is
	given by $F^a_t$. 
	\begin{proof}
		The theorem follows from Theorem 5.1 in \cite{SOP}  and Corollary \ref{cor:bigeodesic}.
	\end{proof}
\end{theorem}
 For $t\in [0,1]$, we define $\eta^a_t$ to be the metric on the space $\mathcal{E}(S,\mathcal{T})_1$ of singular Euclidean structures introduced in \S \ref{s:surface} by  assigning to each pair $(A,B)$ the value $$\eta^a_t(A,B)=(1-t)\eta(A,B)+t\eta(B,A).$$

\begin{theorem}
	For every $t\in [0,1]$, let $F({\mathcal{T}})^a_t$ be the following function  on the tangent space 
	$T(\mathcal{E}(S,\mathcal{T})_1)$ of $\mathcal{E}(S,\mathcal{T})_1$:
	$$ ((A_{ijk}),(v_{ijk}))\mapsto (1-t)\max_{(ijk)}\{\frac{v_{ijk}}{A_{ijk}}\}+t \max_{(ijk)}\{\frac{-v_{ijk}}{A_{ijk}}\}.$$
	The metric $\eta^a_t$  is induced
	by the  Finsler structure on $\mathcal{E}(S,\mathcal{T})_1$
	given by $F({\mathcal{T}})^a_t$. Furthermore,  for any two points of the metric space $(\mathcal{E}(S,\mathcal{T})_1,\eta^a_t)$, there is always a bigeodesic between them.
	\begin{proof}
		The proof follows immediately from Theorem 5.1  in \cite{SOP} and Theorem \ref{cor:bigeodesic-2}.
	\end{proof}
	
\end{theorem}

One may ask several questions regarding the metrics $\eta^a_t$, for instance, whether they are all asymmetric except for $t=1/2$ and whether any two metrics $\eta^a_t$ and $\eta^a_{t'}$ are non-isometric for $t\not=t'$. One may ask similar questions regarding the metrics on $\frak{T}_1$ defined by 
$$\eta^m_t(X,Y)=\max\{(1-t)\eta(X,Y),t\eta(Y,X)\}.$$
Here, the superscript $m$ in $\eta^m_t$ stands for $\max$.

\section{The completion of an asymmetric metric space} \label{s:completion}

In this section, we develop a theory of completion of a metric space which is not necessarily symmetric.

Let $(X,d)$ be a (non-necessarily symmetric)  metric space.

\begin{definition}[Forward and backward Cauchy sequence]
\label{d:symmetry}
	 A sequence $(x_i)$ in $X$ is called a \emph{forward Cauchy sequence} if for every $\epsilon >0$ there exists an integer $N$ such that $d(x_i,x_{i+k})<\epsilon$ for any $i\geq N$ and $k\geq 0$. The sequence $(x_i)$ is called a \emph{backward Cauchy sequence} if for every $\epsilon >0$ there exists an integer $N$ such that $d(x_{i+k},x_i)<\epsilon$ for any $i\geq N$ and $k\geq 0$. The space $X$ is said to be \emph{forward complete} if for every forward Cauchy sequence $(x_n)$, there exists $x\in X$ such that $d(x_n,x)\to 0$ as $n\to \infty$. Similarly, $X$ is said to be \emph{backward complete} if for every backward Cauchy sequence $(x_n)$, there exists $x\in X$ such that $d(x,x_n)\to 0$ as $n\to \infty$.
\end{definition}

\begin{definition}[Convergence-symmetry property]\label{def:c-s}
	A non-necessarily symmetric metric space $(X,d)$ is said to have the convergence-symmetry property, or to be convergence-symmetric,  if for any two sequences $(p_n)$ and $(q_n)$ in $X$ such that $d(p_n,q_n)\to 0$ as $n\to \infty$, we have also
$d(q_n,p_n)\to 0$ as $n\to \infty$. 
\end{definition}

\begin{remark}
 Herbert Busemann introduced in Chapter 1 of his book \emph{Synthetic differential geometry}  \cite[p.\ 2]{Busemann-synthetic}  a notion of convergence property  for (non-necessarily symmetric) metric spaces which is weaker than our convergence-symmetry property. He worked under the condition that for any point $x$ in the space and for any sequence $(p_n)$,  we have $d(x,p_n)\to 0$ if and only if  $d(p_n,x)\to 0$.
With such a condition, he proved theorems such as the Hopf--Rinow theorem for metric spaces which possibly do not satisfy the symmetry axiom, but this did not allow him to define the completion of such a general metric space, as we do in the present paper.

The notions of completeness and completion without the convergence-symmetric property were studied in Algom--Kfir \cite{AK} and Appendix B of Huang--Ohshika--Pan--Papadopoulos \cite{HOPP} in settings which are more complicated than ours here.
\end{remark}

\begin{lemma}
	Assume that $(X,d)$ has  the convergence-symmetry property. Then a sequence in $X$ is forward  Cauchy if and only if it is backward Cauchy.
\end{lemma}
	\begin{proof}
Let $(x_n)$ be a forward Cauchy sequence in $X$. 
Assume that $(x_n)$ is not backward Cauchy. 
Then we can find $\epsilon>0$ such that    for all $N>0$ there exist integers $m(N)\geq n(N)\geq N$ such that $d(x_{m(N)},x_{n(N))})\geq\epsilon$. Now since $(x_n)$ is forward Cauchy, we have $d(x_{n(N)},x_{m(N)})\to 0$ as $N\to \infty$. 
From the convergence-symmetry property we have $d(x_{m(N)},x_{n(N))})\to 0$ as well. 
This is a contradiction. Therefore every forward Cauchy sequence is backward Cauchy. By the same proof, the converse holds.  
	\end{proof}
	
Therefore in a metric space having the convergence-symmetry property, we can simply call a sequence Cauchy if it is forward (or equivalently backward) Cauchy.

\begin{definition}[Convergent sequence in a convergence-symmetric space]
	If $(X,d)$ is a metric space having the convergence-symmetry property, then for every sequence $(x_n)$ in $X$, we have $d(x_n,x)\to 0$ if and only $d(x,x_n)\to 0$ as $n\to \infty$. If this holds, we say that the sequence $(x_n)$ converges to $x$. Also observe that $(x_n)$ is Cauchy if and only if for each $\epsilon >0$ there is $N$ such that for all $n,m\geq N$ we have $d(x_n,x_m)< \epsilon$.	
\end{definition}

\begin{definition}[Complete convergence-symmetric space]
	A metric space having the convergence-symmetry property is called complete if every Cauchy sequence is convergent.
\end{definition}

\begin{definition}[Dense subset in a convergence-symmetric space]
	Let $(X,d)$ be a convergence-symmetric metric space. A subset $A$ of $X$ is called dense in $X$ if for each $x\in X$ there is a sequence $(x_n)$ in $A$ converging to $x$.
\end{definition}

 Let $(X,d)$ be a convergence-symmetric metric space. We  construct a completion of $X$ with respect to $d$. The outline is the same as the usual one followed in the case of symmetric metric spaces. We point out the steps which depend crucially on the fact that our space satisfies the convergence-symmetry property.
 
 (1)
 Consider the set of Cauchy sequences in $X$. We say that two Cauchy sequences $(p_n)$ and $(q_n)$ are equivalent if 
 $$d(p_n,q_n)\to 0\ \text{as} \ n\to \infty.$$
 It is clear from the definition that this defines  an equivalence relation on the set of Cauchy sequences in $X$.
 
 (2)
 If $(p_n)$ and $(q_n)$ are Cauchy sequences, then for any $n,m\geq 0$ we have 
 
 $$d(p_n,q_n)\leq d(p_n,p_m)+d(p_m,q_m)+d(q_m,q_n)$$
 It follows that the sequence of real numbers $(d(p_n,q_n))$ is a Cauchy sequence.  Thus it is convergent. 
 
 (3)
 Let $X^*$ be the set of equivalence classes of Cauchy sequences in $X$. If $P\in X^*$ and $Q\in X^*$ denote the equivalence classes of $(p_n)$ and $(q_n)$, then we define 
 $$\Delta(P,Q)=\lim_{n\to \infty }d(p_n,q_n).$$
 This limit exists since $(p_n)$ and $(q_n)$ are both forward and backward Cauchy.
 Also if $(p'_n)$
 and $(q'_n)$ are other representatives of $P$ and $Q$, then we have
 $$d(p'_n,q'_n)\leq d(p'_n,p_n)+d(p_n,q_n)+d(q_n,q'_n),$$
  therefore the two sequences $(d(p'_n,q'_n))$ and $(d(p_n,q_n))$ have the same limit. It follows that $\Delta$ is well defined. 
 
 (4)
 It is easy to see that $\Delta$ is a function satisfying the  axioms of a metric space.
 
 (5)
 We show that $(X^*, \Delta)$ has the convergence-symmetry property. 
 
 Let $(P_n)$ and $(Q_n)$ be two sequences in $X^*$ such that $\Delta(P_n,Q_n)\to 0$ as $n\to\infty$. 
 Let $(p_{n,m})_m$ and $(q_{n,m})_m$ be representatives for $P_n$ and $Q_n$, respectively. Then we have 
 $\lim_{n\to\infty}\lim_{m\to\infty}d(p_{n,m},q_{n,m})=0,$
 which means that for every $\epsilon >0$, there exist $N$ and $M(n)$ such that if $n\geq N$ and $m \geq M(n)$, then $d(p_{n,m}, q_{n,m}) < \epsilon$.
 This is equivalent to saying that there are $n(k)$ and $m(k)$ for $k \in \mathbb{N}$ such that for every $i(k)\geq n(k)$ and $j(k) \geq m(k)$, we have $\lim_{k \to \infty}d(p_{i(k), j(k)}, q_{i(k), j(k)}) \to 0$.
 Since $X$ has the convergence-symmetry property, this implies that $\lim_{k \to \infty} d(q_{i(k), j(k)}, p_{i(k), j(k)})=0$ for every $i(k) \geq n(k)$ and $j(k) \geq m(k)$.
 By the above equivalence, this implies that $\lim_{n\to\infty}\lim_{m\to\infty}d(q_{n,m},p_{n,m})=0,$ which means that $\Delta(Q_n, P_n) \to 0$.
 
 (6)
 We prove that $X^*$ is complete. 
 Since we already showed that $X^*$ is convergence-symmetric, we have only to prove that $X^*$ is forward complete.
 Let $(P_n)$ be a Cauchy sequence in $X^*$, and $(p_{n,m})_m$ a sequence which is a representative of $P_n$.
 Since $(P_n)$ is forward Cauchy, for any $k \in \mathbb{N}$, there exists an integer  $N(k)$ such that for any $n \geq N(k)$ and any non-negative integer $i$, we have $\Delta(P_{n+i}, P_n) < 1/k$.
 Without loss of generality, we can take $N(k)$ so that it is monotone increasing with respect to $k$.
 This means that for any $n\geq N(k)$, there is $M(n,i)$ such that for any $m \geq M(n,i)$, we have $d(p_{n+i, m}, p_{n, m}) < 1/k$.
 On the other hand, since each $(p_{n,m})$ is a forward Cauchy sequence, there is $M(n,k)$ such that for any $m \geq M(n,k)$ and any non-negative integer $i$, we have $d(p_{n, m+i}, p_{n,m}) < 1/k$.
 Again we can take $M(n,k)$ so that it is monotone increasing with respect to both $n$ and $k$.
 
 Now we consider the sequence $Q=(q_k):=(p_{N(k),M(N(k), k)})_k$.
 We can see that this sequence is forward Cauchy.
 Indeed for any non-negative integer $i$, we have 
 \begin{equation*}
 	\begin{split}
 		&d(p_{N(k+i), M(N(k+i), k+i)}, p_{N(k), M(N(k), k)}) \\
 		&\leq d(p_{N(k+i), M(N(k+i), k+i)}, p_{N(k), M(N(k+i), k+i)})+d(p_{N(k), M(N(k+i), k+i)}, p_{N(k), M(N(k),k)})\\& < 2/k.
 	\end{split}
 \end{equation*}
 
 We next show that $d(P_n, Q) \to 0$.
 For any $k$, we can choose $K$ such that if $n > K$, then $n>N(k)$.
 Then we have $\Delta(P_n, P_{N(k)}) < 1/k$, which means that there is $L(k)$ such that for any $m \geq L(k)$, we have $d(p_{n,m}, p_{N(k), m}) < 1/k$.
 We can assume that $L(k)\geq M(N(k), k)$, by replacing $L(k)$ with a bigger one if necessary.
 Then we also have $d(p_{N(k), m}, p_{N(k), M(N(k), k)})<1/k$,  hence $d(p_{n,m}, q_k)<2/k$.
 Since $Q$ is also backward Cauchy, for sufficiently large $k$, we can take a larger $m$ so that $d(q_k, q_m)< 1/k$, hence we have $\lim_{m\to \infty} d(p_{n,m}, q_m) \leq 3/k$.
 Thus we have shown that $\Delta(P_n, Q) \to 0$.
 
 (7)
 For each $p\in X$ consider the constant sequence all of whose terms are $p$. This is evidently a Cauchy sequence. Let $P_p$ be the element of $X^*$ which contains this sequence. It is obvious that $\Delta(P_p,P_q)=d(p,q).$
 Therefore there is a distance-preserving map 
 $\phi: X\to X^*$
 given by $\phi(p)=P_p$.
 
 (8)
 It is not difficult to see that $\phi(X)$ is dense in $X^*$. Indeed let $P\in X^*$ be represented by the sequence $(p_n)$. Let $P_n=Pp_n$. Then $P_n\to P$ as $n \to \infty$.
 
 (9)
 Let $(X',d')$ be a  metric space having the convergence-symmetry property and assume that it is complete.  
 Suppose that  there is a distance-preserving map $\phi':X\to X'$.
 Then $\phi'$ can be extended to a map $X^*\to X'$ by sending 
 an element $P$ represented by a Cauchy sequence $(p_n)$ to the unique limit  of $(\phi'(p_n))$ in $X'$.
 This means that $X^*$ is the minimal complete convergence-symmetric metric space containing $X$.

The following proposition will be useful. It gives a relation between the completeness of a metric space having the convergence-symmetry property and the completeness in the usual sense of a symmetric metric space.
\begin{proposition} \label{prop-crucial}
	Assume that $(X,d)$ is a metric space having the convergence-symmetry property. Let $d_{\mathrm{arith}}$ and $d_{\mathrm{max}}$ be the following metrics:
	$$d_{\mathrm{arith}}(x,y)=\frac{1}{2}(d(x,y)+d(y,x))$$
	$$d_{\mathrm{max}}=\max\{d(x,y), d(y,x)\}.$$
	Then the following three conditions are equivalent.
	\begin{enumerate}
		\item 
		$(X,d)$ is complete.
		\item 
		$(X,d_{\mathrm{arith}})$ is complete.
		\item 
		$(X,d_{\mathrm{max}})$ is complete.
		
	\end{enumerate}
	\end{proposition}
	\begin{proof}
		Since the symmetric metrics $d_{\mathrm{arith}}$ and $d_{\mathrm{max}}$ are bi-Lipschitz equivalent, it is clear that Conditions (2) and (3) are equivalent.   Assume that $(X,d)$ is complete. Let $(x_n)$ be a Cauchy sequence in $X$ with respect to the metric $d_{\mathrm{arith}}$. Then $(x_n)$ is Cauchy with respect to $d$. It follows that $(x_n)$ (both forward and backward) converges  to a point $x \in X$ with respect to the metric $d$. 
Therefore $(x_n)$ converges to $x$ with respect to the metric $d_{\mathrm{arith}}$. Hence $(X,d_{\mathrm{arith}})$ is complete. 

Now assume that $(X,d_{\mathrm{arith}})$ is complete. Let $(x_n)$ be a Cauchy sequence with respect to the metric $d$. By the convergence-symmetry property, this sequence is a Cauchy sequence with respect to $d_{\mathrm{arith}}$, hence $(x_n)$ converges to $x$ with respect to $d_{\mathrm{artih}}$, for some $x\in X$. This also implies that $(x_n)$ converges to $x$ with respect to $d$. Therefore, $(X,d)$ is complete.
	\end{proof}

\begin{remark}
In the paper \cite{Mennucci2014},  Mennucci studies geodesics in the context of asymmetric metric spaces. However he does not consider the notion of completeness or completion for such a space.
 \end{remark}
 
\section{On the topology defined by the metric $\eta$} \label{s:topology}
Let $S$ be a surface and $\mathcal{T}$ a triangulation of $S$. In this section, we study a topology on $\mathcal{E}(S,\mathcal{T})_1$ induced by the metric $\eta$.

\begin{proposition}
Let $A=(A_{ijk})\in \mathcal{E}(S,\mathcal{T})_1$ and $A(n)=(A_{ijk}(n))$ be respectively a point and a sequence in  $\mathcal{E}(S,\mathcal{T})_1$. 
Then, 
\begin{enumerate}
\item
$\eta(A(n),A)\to 0$ if and only if $A_{ijk}(n)\to A_{ijk}$ for all $(i,j,k)\in \mathcal{J}$.
\item
$\eta(A,A(n))\to 0$ if and only if $A_{ijk}(n)\to A_{ijk}$ for all $(i,j,k)\in \mathcal{J}$.
\end{enumerate}
\end{proposition}

\begin{proof}
We only prove the first part. The proof of the second part is quite similar.
If $A_{ijk}(n)\to A_{ijk}$ for every $(i,j,k) \in \mathcal J$, then $\max\{A_{ijk}/A_{ijk}(n)\}\to 1$, therefore $\eta_{\mathcal{T}}(A(n),A)\to 0$.  

For the other implication, assume that $\eta(A(n),A)\to 0$. Then $\max\{A_{ijk}/A_{ijk}(n)\}\to 1$. 
Suppose, seeking a contradiction, that there exists $(i',j',k')\in \mathcal{J}$ such that 
$\lim_{n\to\infty}\frac{A_{i'j'k'}(n)}{A_{i'j'k'}}\neq 1$. 
Passing to subsequences, we may assume that
\begin{enumerate}
\item
$\lim_{n\to\infty}A_{ijk}(n)$ exists for each $(i,j,k)$; and
\item
$\lim_{n\to \infty }A_{i'j'k'}(n)< A_{i'j'k'}$.
\end{enumerate} 
Then,  by (\ref{f:Heron}), we have
$$1=\mathrm{Ar}(A)>\mathrm{Ar}(\lim_{n\to \infty}\{A_{ijk}(n)\})=1,$$
which is a contradiction.
\end{proof}

This proposition immediately implies the following.

\begin{corollary}
$\eta(A(n),A)\to 0$ if and only if $\eta(A,A(n))\to 0$.
\end{corollary}

Consider the following metrics:

\begin{enumerate}
\item
$\eta_{\mathrm{arith}}(A,B)=\frac{1}{2}(\eta(A,B)+\eta(B,A))=\eta^a_{1/2}(A,B)$.
\item
$\eta_{\mathrm{max},}(A,B)=\max\{\eta(A,B),\eta(B,A)\}$.

\end{enumerate}
These two metrics are bi-Lipschitz equivalent. By the topology induced by $\eta$, we mean the topology induced by $\eta_{\mathrm{arith}}$ or $\eta_{\mathrm{max}}$ (the two are equivalent).

\begin{remark}
Let $A=(A_{ijk})$ and $A(n)=(A_{ijk}(n))$ be in $\mathcal{E}(S,\mathcal{T})$. Then, the following are equivalent:
\begin{enumerate}
\item
$A_{ijk}(n)\to A_{ijk}$ for each $(i,j,k)$.
\item
$\eta_{\mathrm{max}}(A(n),A)\to 0$.
\item
$\eta_{\mathrm{arith}}(A(n),A)\to 0$.
\end{enumerate}
It follows that the usual topology of $\mathcal{E}(S,\mathcal{T})$ as a subspace of $\R^{\mathcal{J}}$ and the topology induced by
$\eta_{\mathrm{max}}$ (or equivalently $\eta_{\mathrm{arith}}$) are the same.
\end{remark}

\begin{corollary}
$\eta^a_t(A(n),A)\to 0$ if and only if $\eta^a_t(A,A(n))\to 0$.
\begin{proof}
	The case where $t=0$ or $t=1$ has been dealt with above. If $t\in(0,1)$, then the claim is obvious. 
	\end{proof}
\end{corollary}

\section{Completeness of the space $\frak{T}_1$} \label{s:completeness}

We shall use the following elementary lemma.

\begin{lemma}
\label{simple-inequality}
Let $a_1, a'_1,\dots a_n, a'_n$ be positive real numbers. Then
$$\frac{a'_1+\dots +a'_n}{a_1+\dots+a_n}\leq \max_i\{\frac{a'_i}{a_i}\}.$$
\end{lemma}
\begin{proof}
We have 
$$a'_1+a'_2+\dots +a'_n=a_1\frac{a'_1}{a_1}+\dots+a_n \frac{a'_n}{a_n}\leq \max_i\{\frac{a'_i}{a_i}\}(a_1+\dots +a_n),$$
\noindent and the result follows.
\end{proof}

\begin{proposition}
	Let $A(n)=(A_1(n),A_2(n),A_3(n)))$ and $A'(n)=(A'_1(n),A'_2(n),A'_3(n))$ be two  sequences in  $\frak{T}_1$. We have
		$$\eta(A(n),A'(n))\to 0\ \text{ if and only if}\ \log A'_{i}(n)-\log A_{i}(n)\to 0\ \text{ for all} \ i=1,2, 3.$$
	
\end{proposition}
\begin{proof}
If $\log A'_i(n)-\log A_i(n)\to 0$, then $A_i'(n)/A_i(n)\to 1$ for each $i$. It follows that $\eta(A(n),A'(n))\to 0$.

For the other implication, assume that $\eta(A(n),A'(n))\to 0$. It follows that $\max\{A'_{i}(n)/A_i(n)\}\to 1$.
Seeking a contradiction, we assume that there exists $i'$ such that $\lim_{n\to \infty}\frac{A'_{i'}(n)}{A_{i'}(n)}\neq 1$. 
Then we may assume the following.
\begin{enumerate}
\item
$\lim_{n\to \infty}\frac{A'_1(n)}{A_1(n)}=1$.
\item
$\lim_{n\to\infty}\frac{A'_2(n)}{A_2(n)}$ exists and is strictly less than $1$.
\item
$\lim_{n\to \infty}\frac{A'_3(n)}{A_3(n)}$ exists and is less than or equal to $1$.
\end{enumerate}

\noindent From Lemma \ref{simple-inequality}, we see that

$$1=\frac{[\mathrm{Ar}(A'(n))]^2}{[\mathrm{Ar}(A(n))]^2}=\frac{A'_1(n)A'_2(n)A'_3(n)(A'_1(n)+A'_2(n)+A'_3(n))}{A_1(n)A_2(n)A_3(n)(A_1(n)+A_2(n)+A_3(n))}\leq\frac{A'_1(n)A'_2(n)A'_3(n)}{A_1(n)A_2(n)A_3(n)}\max_i\{\frac{A'_i(n)}{A_i(n)}\}$$
 Letting $n\to \infty$,  we reach a contradiction.
\end{proof}
	
\begin{corollary}
The space $\frak{T}_1$ has the convergence-symmetry property.
\end{corollary}

\begin{theorem}\label{thm:complete}
$\frak{T}_1$ is complete with respect to the metric $\eta$.		
\end{theorem}
	\begin{proof}
By Proposition \ref{prop-crucial} it is sufficient to show the metric $\eta_{max}(A,B)=\max\{\eta(A,B),\eta(B,A)\}$ is complete since $\eta$ has the convergent-symmetry property.  
We note that $\eta_{max}$ is the restriction of the metric on $(\R_+^*)^3$ given by
$$d((A_1,A_2,A_3),(A'_1,A'_2,A'_3))=\max_{1\leq i \leq 3}\{\lvert \log A'_{i}-\log A_{i}\rvert \}.$$
Now consider the metric $d_1$ given on $\R^3$ by the following formula:
$$d_1((B_1,B_2,B_3),(B'_1,B'_2,B'_3))=\max_{1\leq i \leq 3}\lvert B'_{i}-B_{i}\rvert.$$
Then it is easy to see that 
$$E: (B_1,B_2,B_3)\to (e^{B_1},e^{B_2},e^{B_3})$$
is an isometry between $(\R^3,d_1)$ and $((\R_+^*)^3,d)$. Note that the former metric space is complete and   bi-Lipschitz equivalent to the usual Euclidean metric. Now since $E^{-1}(\mathcal{T}_1)$ is a closed subset of $\R^3$, the result follows.
\end{proof}

\begin{remark}
	For each $t\in (0,1)$, the space $(\frak{T}_1,\eta^a_t)$ has the convergence-symmetry property.
\end{remark}

\begin{theorem}
For each $t\in (0,1)$, the space	$\mathcal{E}(S,\mathcal{T})_1$ is complete with respect to the metric $\eta^m_t$, where 
	$$\eta_t^m(A,B)=\max\{(1-t)\eta(A,B), t\eta(B,A) \}.$$ Thus it follows that $\eta^a_t$ is complete as well.
	\end{theorem}
	\begin{proof}
		The proof is  similar to that of Theorem \ref{thm:complete}. We can assume without loss of generality that $t\geq 1-t$. By Proposition \ref{prop-crucial}, it suffices to show that the metric $$\eta^m_{t,max}(A,B)=\max\{\eta_t^m(A,B),\eta_t^m(B,A)\}=t\max\{ \eta(A,B),\eta(B,A))\}$$
		  is complete,  since $\eta_t^m$ has the convergent-symmetry property.
		Observe  that $\eta^m_{t,max}$ is the restriction of the metric on $(\R_+^*)^{\mathcal{J}}$ given by
		$$d((A_{ijk}),(A'_{ijk}))=t\max_{ijk}\{\lvert \log A'_{ijk}-\log A_{ijk}\rvert \}.$$
		Now consider the metric $d_1$ on $\R^{\mathcal{J}}$ given by the following formula:
		$$d_1((B'_{ijk}),(B_{ijk}))=t\max_{ijk}\lvert B'_{ijk}-B_{ijk}\rvert.$$
		Then it is not difficult to see that 
		$$E: (B_{ijk})\to (e^{B_{ijk}})$$
		is an isometry between $(\R^{\mathcal{J}},d_1)$ and $((\R_+^*)^{\mathcal{J}},d)$. Note that the former metric space is complete and   bi-Lipschitz equivalent to the usual Euclidean metric. Now since $E^{-1}(\mathcal{E}(S,\mathcal{T})_1)$ is a closed subset of $\R^{\mathcal{J}}$, the result follows.
	\end{proof}

\begin{example}
\label{incomplete}
 {\rm In this example, the surface is a disc and 
	  $\mathcal{T}$ is a triangulation of the disc by two  triangles.
	We show that $\mathcal{E}(S,\mathcal{T})_1$ does not have the convergence-symmetry property and that it is forward incomplete with respect to the metric $\eta(\mathcal{T})$. 
	
Set $c(n)=\sqrt{\sqrt{3}+1/n}$, $d(n)=\sqrt{\sqrt{3}+2/n}$, $a(n)=\sqrt{1+1/n^2}$ and $b(n)=\sqrt{1+4/n^2}$. Let $Q_n$ be the quadrangle with edge lengths 
$$\frac{a(n)}{c(n)},\frac{a(n)}{c(n)},\frac{2}{c(n)},\frac{2}{c(n)}$$
and with a diagonal of length $\frac{2}{c(n)}$ dividing $Q_n$ into two isosceles triangles. Note that the area of $Q_n$ is equal to 1. Similarly let $Q'_n$ be the quadrangle with edge lengths 
$$\frac{b(n)}{d(n)},\frac{b(n)}{d(n)},\frac{2}{d(n)},\frac{2}{d(n)}$$
and with a diagonal of length $\frac{2}{d(n)}$ dividing $Q'_n$ into two isosceles triangles. Note that the area of $Q_n$ is also equal to 1. Therefore, $Q_n$ and $Q'_n$ lie in $\mathcal{E}(S,\mathcal{T})_1$. 
Now it is not difficult to see that 
$$\eta(Q_n,Q'_n)\to \log 4 \ \text{and}\  \eta(Q'_n,Q_n)\to 0 \ \text{as}\ n\to \infty,$$
which shows that $\mathcal{E}(S,\mathcal{T})_1$ does not have the convergence-symmetry property.	

Now we can also see that the sequence $(Q_n)$ is forward Cauchy, whereas in the limit the triangle with sides $a(n)/c(n), a(n)/c(n), 2/c(n)$ degenerates to a segment.
This shows that the space $\mathcal{E}(S,\mathcal{T})_1$ is not forward complete.}
\end{example}

\bigskip

\noindent
\.{I}smail Sa\u{g}lam, Adana Alparslan Turkes Science and Technology University,\\ Department of Aerospace Engineering, 
Adana,T\"{u}rkiye, 
\\
e-mail: isaglamtrfr@gmail.com\\

\noindent Ken'ichi Ohshika,
Department of Mathematics,
Gakushuin University,
Mejiro, Toshima-ku, 171-8588 Tokyo, Japan,
\\
and
\\
Max-Planck-Institut für Mathematik,
Vivatsgasse 7, 53111 Bonn, Germany
\\
 e-mail: 
  \email{ohshika@math.gakushuin.ac.jp}\\

\noindent Athanase Papadopoulos, 
 Institut de Recherche Math\'ematique Avanc\'ee, 
 CNRS et Universit\'e de Strasbourg, 
7 rue Ren\'e Descartes, 67084, Strasbourg, France 
\\
and
\\
Max-Planck-Institut für Mathematik,
Vivatsgasse 7, 53111 Bonn, Germany
\\
 e-mail: 
   \email{papadop@math.unistra.fr}


\end{document}